\documentclass[a4paper,11pt]{amsart}
\usepackage{hyperref,latexsym}
\usepackage{enumerate}
\usepackage{graphicx}
\usepackage{color}
\usepackage{amsfonts}
\usepackage{amssymb}
\usepackage{amsmath}
\usepackage{tikz}
\usepackage{tikz-cd}
\usepackage{caption}
\usepackage[toc,page]{appendix}
\usepackage{listings}
\usepackage{pdfpages}
\usepackage{mathtools}
\usepackage[all,cmtip]{xy}
\usepackage{lipsum}
\usepackage{breqn}
\usepackage{subcaption}
\usepackage{hyperref}
\usepackage{tkz-euclide}
\usepackage[sort,nocompress]{cite}
\usetikzlibrary{angles,quotes}
\usetikzlibrary{decorations.markings}
\usetikzlibrary{hobby}

\def\Z{\mathbb{Z}}

\def\R{\mathbb{R}}

\def\<={\leq}
\def\>={\geq}

\theoremstyle{plain}
\newtheorem{theorem}{Theorem}[section]

\newtheorem{proposition}[theorem]{Proposition}

\theoremstyle{definition}
\newtheorem{definition}[theorem]{Definition}

\theoremstyle{remark}
\newtheorem*{remark}{Remark}
\newtheorem*{remarks}{Remarks}

\newcommand{\set}[1]{\{#1\}}
\newcommand{\bigset}[1]{\Big\{#1\Big\}}

\renewcommand{\pod}[1]{\allowbreak\mathchoice
  {\if@display \mkern 18mu\else \mkern 8mu\fi (#1)}
  {\if@display \mkern 18mu\else \mkern 8mu\fi (#1)}
  {\mkern4mu(#1)}
  {\mkern4mu(#1)}
}

\begin{document}

\title[]
      {Locally maximizing orbits for the non-standard generating function of convex billiards and applications
     }

\date{}
\author{Misha Bialy}
\address{School of Mathematical Sciences, Raymond and Beverly Sackler Faculty of Exact Sciences, Tel Aviv University,
Israel} 
\email{bialy@tauex.tau.ac.il}
\thanks{MB is partially supported by ISF grant 580/20,  D.T is supported by ISF grants 580/20, 667/18 and DFG grant MA-2565/7-1  within the Middle East Collaboration Program.}

\author{Daniel Tsodikovich}
\address{School of Mathematical Sciences, Raymond and Beverly Sackler Faculty of Exact Sciences, Tel Aviv University,
	Israel}
\email{tsodikovich@tauex.tau.ac.il}


\begin{abstract}
Given an exact symplectic map $T$ of a cylinder with a generating function $H$ satisfying the so-called negative twist condition, $H_{12}>0$, we study the locally maximizing orbits of $T$, that is, configurations which are local maxima of the action functional $\sum_n H(q_n,q_{n+1})$.
We provide a necessary and sufficient condition for a configuration to be locally maximizing.
Using it, we consider a situation where $T$ has two generating functions with respect to two different sets of symplectic coordinates.
We suggest a simple geometric condition which guarantees that the set of locally maximizing orbits with respect to both of these generating functions coincide.
As the main application we show that the two generating functions for planar Birkhoff billiards satisfy this geometric condition.
We apply it to get the following result: consider a centrally symmetric curve $\gamma$, for which the Birkhoff billiard map has a rotational invariant curve $\alpha$ of $4$-periodic orbits.
We prove that a certain $L^2$-distance between $\gamma$ and its ``best approximating" ellipse can be bounded from above in terms of the measure of the complement of the set filled by locally maximizing orbits lying between $\alpha$ and the boundary of the phase cylinder. 
Moreover, this estimate is sharp, giving an effective version of a recent result on 
Birkhoff conjecture for centrally symmetric curves \cite{bialy2020birkhoffporitsky}.
We also get a similar bound for arbitrary curves $\gamma$, that relates the measure of the complement of the set of locally maximizing orbits with the $L^2$-distance between $\gamma$ and its ``best approximating" circle.
\end{abstract}

\maketitle

\section{Introduction and the results}
\subsection{\bf Twist maps}
Twist maps of the cylinder arise in the study of various dynamical systems, including mathematical billiards. 
In the case of exact twist maps, there exists a generating function, which allows to investigate the orbits of a twist map with a variational approach \cite {AubryS1983TdFm, MatherJohnN1991Vcoo, Bangert1988MatherSF, MacKayR.S1985CKTa, MacKayRS1989CKtf}, see also \cite {SiburgKarlFriedrich2004Tpol} and \cite{arnaud2010green}.

Our motivation comes from the billiard map in strictly convex domains with $C^2$ smooth boundary.
This dynamical system can be described as a twist map with respect to two different generating functions, so that it is a negative twist map with respect to both.
 The first generating function is the length of the chord (used extensively by Birkhoff).
 The second, non-standard generating function which is related to the support function of the table, was invented in \cite{BialyMisha2017Abaa}.
Recently, this non-standard generating function was found to be useful in \cite{BialyMisha2020DRia, BialyMisha2017Abaa, BialyMisha2018Gbti}.

Our goal is to compare variational properties of the action functionals for these two generating functions, but we formulate the results for general twist maps.

Consider an exact symplectic twist map $T$ of the cylinder $\mathbb{A}=S^1\times\mathbb{R}$ with symplectic coordinates $(q,p)$.
Denote by $H(q,q')$ the generating function.
 We shall assume, non-traditionally, that the twist is negative, that is 
\[
H_{12} (q,q') >0.
\]  
For the function $H$ we define the variational principle as follows. For the configuration sequence $\{q_n\}$ we associate the formal sum \[\sum_n H(q_n,q_{n+1}).\]
We consider {\it locally maximizing} configurations, that is, those configurations which give {\it local maximum}
for the functional between any two end-points.
We shall call such configurations, \textit{m-configurations}, and the corresponding orbits on the phase cylinder $\mathbb{A}$, \textit{m-orbits}, see Section \ref{section:analysis} for the precise definitions.
We denote by $\mathcal M_H\subset\mathbb{A}$ the set swept by all m-orbits corresponding to the variational principle for the generating function $H$.

In our first result, we summarize the ideas of \cite{Bialy1993, Michael2012Hrfc, MacKayR.S1985CKTa} and prove the following characterization of m-orbits.

\begin{theorem}\label{thm:mconfigPosJacobi}
	Let $T:\mathbb{A}\to\mathbb{A}$ be an exact twist map with the generating function $H$, satisfying the twist condition $H_{12}>0$. Let $\set{q_n}$ be a configuration sequence corresponding to the orbit $\set{(q_n,p_n)}$.
	Then the orbit $\set{(q_n,p_n)}$ is an m-orbit if and only if there exists a positive Jacobi field along $\{q_n\}$.
\end{theorem}

Motivated by the example of convex billiards, we shall consider an exact symplectic map of the cylinder $\mathbb{A}$ which is a twist map with respect to two sets of symplectic coordinates $(q,p)$ and $(x,y)$, with generating functions $H,G$ respectively.
 It is natural to ask whether the sets $\mathcal{M}_H$ and $\mathcal{M}_G$ are the same.

Our method to compare the m-orbits with respect to the generating functions $H,G$ relies on the following geometric picture.
At every point $z\in\mathcal{M}_H$ one can partition the tangent space $T_z \mathbb{A} $ to four cones that are determined by the image and pre-image (by $T$) of the vertical direction $\frac{\partial}{\partial p}$, see Definition \ref{def:2211}, and Figure \ref{fig:coneCondition}.
The relation of these cones with variational properties of orbits was first studied in \cite{MacKayR.S1985CKTa}.

We denote by $N_H$ the ``north" cone.
Similarly, for every point $z\in\mathcal{M}_G$ one can partition the tangent space $T_z \mathbb{A}$ to four cones that are determined by the image and pre-image of the vertical direction $\frac{\partial}{\partial y}$, and we denote the ``north" cone by $N_G$.
We shall assume the following: 

{\bf Geometric assumption.} 

 	\begin{equation}\label{eq:assumption}
 	\begin{cases}
 		z\in \mathcal{M}_H \implies \frac{\partial}{\partial y}(z)\in N_H ,\\
 		z\in\mathcal{M}_G \implies \frac{\partial}{\partial p}(z)\in N_G .
 	\end{cases}
 \end{equation}
\begin{figure}
	\centering
	\begin{tikzpicture}
		
		\draw (-3,0)--(3,0);
		\draw(0,-3)--(0,3);
		\draw[dashed] (-2,-2)--(2,2);
		\draw[dashed] (-2,2)--(2,-2);
		\draw[->, line width = 2](0,0)--(0,1.5);
		
		\draw[blue] (-1,3)--(1,-3);
		\draw[->,blue, line width = 2](0,0)--(-0.5,1.5);
		
		\node at (0.5,3) {$N_H$};
		\node at (-0.5,-3) {$S_H$};
		\node at (-3,0.3) {$W_H$};
		\node at (3,0.3) {$E_H$};
		
		\node at (0.5,1.5) {$\frac{\partial}{\partial p}$};
		\node[blue] at (-0.3,2) {$\frac{\partial}{\partial y}$};
	\end{tikzpicture}
	\caption{Cone condition. The black solid lines represent the coordinate axes, and the dashed lines represent the lines that define the cones with respect to $(q,p)$ coordinates. The vertical direction is marked. The blue line is the vertical direction with respect  to $(x,y)$ coordinates. \label{fig:coneCondition}}
\end{figure}
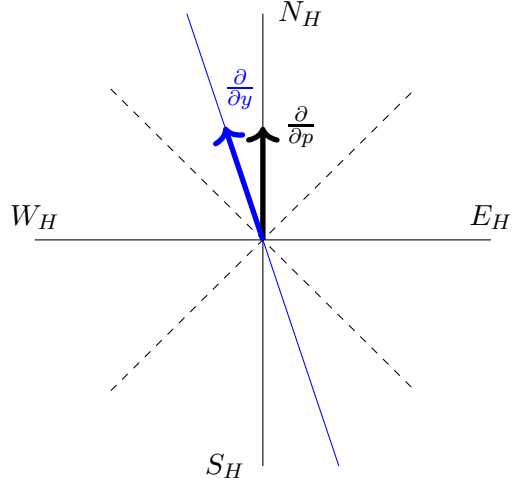
We prove the following:\begin{theorem}\label{thm:equalMOrbits}

	Let $T:\mathbb{A}\to\mathbb{A}$ be an exact twist map, with respect to two sets of symplectic coordinates, $(q,p)$ and $(x,y)$ and generating functions $H,G$ satisfying the twist condition $H_{12}, G_{12}>0$. 
	Assume the geometric assumption \eqref{eq:assumption} holds.
	Then the sets of m-orbits corresponding to the variational principles for $H$ and $G$ coincide:
\[\mathcal{M}_H=\mathcal{M}_G.\]
\end{theorem}
Theorem \ref{thm:equalMOrbits} is proved in Section \ref{subsection:twogenfun}.

Moreover we shall prove in  Theorem \ref{thm:equalityForBirkhoff} below that the two generating functions for the Birkhoff billiard map satisfy the geometric assumption \eqref{eq:assumption}, proving that the m-orbits of both those generating functions coincide.

\begin{remarks}
\begin{enumerate}
	\item Our proof of Theorem \ref{thm:equalMOrbits} is geometric and is based on the criterion of Theorem \ref{thm:mconfigPosJacobi}. We do not know if the Geometric assumption \eqref{eq:assumption} is really necessary for $\mathcal{M}_H=\mathcal{M}_G$ to hold.
\item One can prove that the result of Theorem \ref{thm:equalMOrbits} and its application to Birkhoff billiards remain true if we replace the class of m-orbits with a smaller class of globally maximizing orbits.

\end{enumerate}
\end{remarks}
\subsection{\bf Application to Convex billiards}

Following \cite{bialy2020birkhoffporitsky}  we shall use the non-standard generating function for convex billiards in order to study the invariant set occupied by m-orbits.

It was proved in  \cite{BialyMisha2015EbiE} that the measure of the complement of the set of m-orbits $\mathcal{M}$ can be estimated from below
in geometric terms (isoperimetric defect of the billiard curve).
Since the set of m-orbits with respect to both generating functions coincide, we can also use the non-standard generating function, to give another estimate for the measure of this set, in addition to the two provided in \cite{BialyMisha2015EbiE}.
\begin{theorem}\label{thm:effectiveGeneralCurve}
Let $\gamma$ be a planar strictly convex $C^2$ smooth curve with support function $h:[0,2\pi]\to\R$ (with respect to an arbitrary origin in the interior of $\gamma$).
Let $\mathbb{A}$ denote the phase cylinder of the Birkhoff billiard map in $\gamma$, $\mathcal{M}\subseteq\mathbb{A}$ denote the set swept by m-orbits, and $\Delta=\mathbb{A}\setminus\mathcal{M}$.
Then the following estimate holds true
\[\mu(\Delta)\geq \pi^2\beta d^2(h,W),\]
where $0<\beta$ is the minimal curvature of $\gamma$, $W$ denotes the subspace of $L^2[0,2\pi]$ spanned by the functions $\set{1,\cos(\psi),\sin(\psi)}$, and $d(\cdot,W)$ denotes the $L^2$- distance to that subspace.
Moreover, this bound is sharp for circles.
\end{theorem}

The method of the proof can also be adopted to give  an effective version of the rigidity of integrable billiards in centrally symmetric curves that was proven in \cite{bialy2020birkhoffporitsky} in the context of the Birkhoff conjecture. 
Let us denote by $\mathcal{C}$ the set of centrally symmetric, strictly convex, $C^2$ smooth curves, for which the billiard map has an invariant curve $\alpha$ with rotation number $1/4$ consisting of $4$-periodic orbits, as was considered in \cite{bialy2020birkhoffporitsky}.
The class $\mathcal{C}$ is rather big and can be fully characterized in terms of the support function $h$, see Section \ref{section:effectiveBounds} below.
We shall denote by $\mathcal{A}$ the domain bounded by $\alpha$ and the upper boundary of the cylinder, and by $\mathcal{M}$ the subset of $\mathbb{A}$ swept by m-orbits.
We give a lower bound for the measure of the invariant subset $\Delta:=\mathcal{A}\setminus\mathcal{M}$, which is free from m-orbits.
\begin{theorem}\label{thm:estimate-beta}
	Let $\gamma$ be a centrally symmetric, strictly convex, $C^2$ smooth curve for which the billiard map has an invariant curve $\alpha$ with rotation number $1/4$ consisting of $4$-periodic orbits.
	Denote by $h:[0,2\pi]\to\R$ the support function of $\gamma$, with respect to the center of symmetry.
	 Set $\Delta=\mathcal{A}\setminus\mathcal{M}$.
	Then the following estimate for the measure holds true:
	\[
	\mu(\Delta)\geq \frac{25\pi^2}{32}\beta^3 d^2(h^2,U),
	\] 
	where $0<\beta$ is the minimal curvature of $\gamma$,  $U$ is the subspace of $L^2[0,\pi]$ spanned by $\set{1,\cos(2\psi),\sin(2\psi)}$, and $d(\cdot,U)$ is the $L^2$- distance from this subspace.
	Moreover, this bound is sharp for ellipses.
\end{theorem}

\begin{remark}
	 The bounds of Theorems \ref{thm:effectiveGeneralCurve}, \ref{thm:estimate-beta} can be considered as effective versions of the results on Birkhoff conjecture \cite{Bialy1993, bialy2020birkhoffporitsky} by estimating closeness of $\gamma$ to the ``closest" circle (for arbitrary curves) or ellipse (for centrally symmetric curves) in terms of the measure of the set $\Delta$.
	 
	 It would be interesting to understand if there exists an effective version of the results of \cite{KaloshinVadim2018OtlB} on the Birkhoff conjecture.
	 \end{remark}

\section{Analysis of locally maximizing orbits}\label{section:analysis}
\subsection{Criterion for m-orbits}\label{subsection:criterion}
Consider an exact symplectic twist map $T$ of the cylinder $\mathbb{A}=S^1\times \mathbb{R}$ with the symplectic coordinates $(q,p)$.
Denote by $H(q,q')$ the generating function. We shall assume non-traditionally that the twist is negative, that is 
\[
H_{12} (q,q') >0.
\]
We consider the corresponding variational principle. For the sequence $\{q_n\}$ we associate the formal sum \[\sum H(q_n,q_{n+1}).\]
We call a sequence a \textit{configuration} if it is a local extreme point of this functional, which is equivalent to the fact that this sequence can be lifted to an orbit $\set{(q_n,p_n)}$ of $T$.
Corresponding to the sign of the twist we consider locally maximizing configurations, that is those configurations which give {\it local maximum}
for the functional between any two end-points. 
In particular the matrices of second variations, for any $M\leq N$ 
\[W_{MN}=
\begin{pmatrix}
a_{M} & b_{M} & 0\cdots&0 &0 \\
b_{M}& a_{M+1} &b_{M+1}\cdots&0&0  \\
\vdots  & \ddots  & \ddots & \ddots & \vdots \\
0 & \cdots&b_{N-2}&a_{N-1}&b_{N-1}  \\
0 & \cdots& 0&b_{N-1} & a_{N} 
\end{pmatrix},\]
are negative semi-definite.
Here $a_n=H_{22}(q_{n-1},q_n)+H_{11}(q_n,q_{n+1})$, and $b_n=H_{12}(q_n,q_{n+1})$, and $\set{q_n}$ denotes the configuration (here and below the subscripts 1,2 denote partial derivatives with respect to the first or second variable).
We shall call the corresponding orbits of $T$, \textit{ m-orbits}, and denote the subset of the cylinder consisting of all m-orbits by $\mathcal{M} $. Here are some elementary properties of $\mathcal{M}$:
\begin{proposition}\label{prop:propertyM}
\begin{enumerate}[(a)]
	\item\label{itm:containsGlobal} $\mathcal{M}$  contains all globally maximizing orbits, and in particular, all rotational invariant curves as well as cantor-tori.
	\item\label{itm:variationNegDef} If the matrix of second variation of some finite segment of a configuration $\set{q_n}$ is negative semi-definite, then the matrix of second variation of  any proper sub-segment is negative definite. In particular, a configuration $\{q_n\}$ is an m-configuration if and only if any finite segment of $\{q_n\}$ has negative definite second variation.
	\item\label{itm:invarianceClosed} $\mathcal{M}$ is a closed set invariant under $T$.
\end{enumerate}
	\end{proposition}
\begin{proof}
	Claim (\ref{itm:containsGlobal}) is obvious. The fact that any rotational invariant curve consists of m-orbits was proved, e.g., by M.Herman (see \cite{SEDP_1987-1988____A14_0,MacKayRS1989CKtf}).
	 
	 Claim (\ref{itm:variationNegDef}) was proved in \cite{MacKayRS1989CKtf}.
	 
Claim (\ref{itm:invarianceClosed}) follows from the fact that a converging sequence of negative semi-definite matrices converges to a negative semi-definite matrix.
By item (\ref{itm:variationNegDef}), this means that the matrix of second variation for every proper sub-segment of the limiting configuration is negative definite, and hence, this sub-segment is a local maximum.
	\end{proof}
A \textit{Jacobi field} along a configuration $\{q_n\}$ is a sequence $\{\delta q_n\}$ satisfying the
\textit{discrete Jacobi equation}:
\begin{equation}
\label{eq:Jacobi}
b_{n-1}\delta q_{n-1}+a_n
\delta q_n+b_{n}\delta q_{n+1}=0,
\end{equation}
where, as before,
\[a_n=H_{22}(q_{n-1},q_n)+H_{11}(q_n,q_{n+1}),\
b_n=H_{12}(q_n,q_{n+1}).\]

It is important that Jacobi fields $\{\delta q_n\}$  along $\{q_n\}$ are in 1-1 correspondence with $T$-invariant vector fields $\{(\delta q_n,\delta p_n)\} $ along the orbit $\set{(q_n,p_n)}$: any solution to the Jacobi equation $\set{\delta q_n}$ can be lifted to a $T$-invariant vector field $(\delta q_n,\delta p_n)$ where 
\[
\delta p_n=-H_{11}(q_n,q_{n+1})\delta q_n-H_{12}(q_n,q_{n+1})\delta q_{n+1}
,\]
or equivalently, due to the Jacobi equation:
\[
 \delta p_n=H_{22}(q_{n-1},q_{n})\delta q_n+H_{12}(q_{n-1},q_{n})\delta q_{n-1}.
\]
Conversely, if $\set{(\delta q_n, \delta p_n)}$ is a $T$-invariant vector field, then $\set{\delta q_n}$ is a solution to the Jacobi equation.
Now we are in position to prove Theorem \refeq{thm:mconfigPosJacobi}.

\begin{proof}[Proof of Theorem \ref{thm:mconfigPosJacobi}]
	1.($\Rightarrow$): Since Jacobi fields are solutions to \eqref{eq:Jacobi}, the space of Jacobi fields  along $\{q_n\}$ is two dimensional.
	 If $\{q_n\}$ is an m-configuration then by Proposition \ref{prop:propertyM} the matrix of second variation $W_{MN}$ is negative definite for all $M\leq N$.
	  In particular, any non-trivial Jacobi field $\{\xi_n\}$ can vanish not more than for one $n$, since otherwise one of the matrices $W_{MN}$ is singular.
	
	Therefore, for fixed $M<N\in\Z$, the following boundary value problem for the Jacobi field $\xi_n$, \[\begin{cases}\xi _{M}=1,\\ \xi_N=0,\end{cases}\] can be uniquely solved.
	 Denote this Jacobi field by $\{\xi^{M,N}_n\}$.
	   Moreover, it follows from \cite[Lemma 3]{Bialy1993} that there exists a strictly positive limiting Jacobi field:
	\[
	\nu_n=\lim\limits_{N\rightarrow +\infty}\xi^{0,N}_n, \nu_n>0, \forall n.
	\]
	2.($\Leftarrow$): Assume now that there exists a positive Jacobi field along a configuration $\{q_n\}$. 
	Then it follows from the discrete Sturm Separation Theorem \cite[Theorem 7.9]{ElaydiSaber1999Aitd} that any other Jacobi field along $\{q_n\}$ which vanishes at 
	$n=K$ keeps a constant sign for all $n<K$ and the opposite sign for all $n>K$.
	 We need to show that for all $M\leq N\in\Z$, the matrix $W_{MN}$ is negative definite.
	  We follow here the argument from \cite{MacKayR.S1985CKTa}.
	   Assume for simplicity that $M=1$.
	    For the principal minors $M_k$ of the matrix $W_{1N}$ we have the recursion formula
	\[	M_{k+1}=a_{k+1}M_k-b_k^2M_{k-1},	\]
	where by convention $M_0=1, M_{-1}=0$. 
	
	On the other hand consider the Jacobi field $\{\xi_n\}$ such that $\xi_{0}=0$, $\xi_{1}=1$.
	 The recursion formula for $\xi_n$ is given by \eqref{eq:Jacobi}.
	Then we have the formula
	\begin{equation}\label{eq:minors}
	\xi_{k+1}=(-1)^k\frac{M_k}{b_1b_2\cdots b_k}.
	\end{equation}
	Indeed, \eqref {eq:minors} holds true for $k=1$, and then can be verified by induction.
	It follows from \eqref{eq:minors} that the sign of $M_k$ equals $(-1)^k$ for all $k\geq 1$, since all $\xi_k$ are positive for	$k\geq 1$.
	 This proves negative definiteness of the matrix $W_{1N}$.
\end{proof}
	\subsection{Function $\omega$ and the bounds}\label{subsection:omegaBounds}
	Given a point $(q_0,p_0)\in\mathbb{A}$ such that its orbit is an m-orbit $\{(q_n,p_n)\}_{n\in\Z}$ we define, following \cite{Bialy1993}, a function \[\omega(q_0,p_0):=-H_{11}(q_0,q_{1})-H_{12}(q_0,q_{1})\nu_1 ,\]
	where $\nu_1=\lim\limits_{N\rightarrow +\infty}\xi^{0,N}_1 $, using the notation of the proof of Theorem \ref{thm:mconfigPosJacobi}. 
	The function $\omega$ is a measurable function on $\mathcal{M}$ as a limit of continuous functions.
	 Moreover, since $\nu_n$ is a positive Jacobi field along $\{q_n\}$ with $\nu_0=1$ we have 
	\[	\omega(q_0,p_0)=H_{22}(q_{-1},q_{0}) +H_{12}(q_{-1},q_{0})\nu_{-1}.\]
	
 Shifting by 1 in the last formula we can write:
	\begin{equation}\label{eq:omega1}
	\begin{cases}
	\omega(T(q_0,p_0))=H_{22}(q_0,q_1)+H_{12}(q_0,q_1)\nu_1(q_0,p_0)^{-1},\\
	\omega(q_0,p_0)=-H_{11}(q_0, q_1)-H_{12}(q_0,q_1)\nu_{1}(q_0,p_0).
	\end{cases}
	\end{equation} 
	Therefore by the twist condition we have the bounds
\begin{equation}\label{eq:bounds}
	H_{22}(q_{-1},q_{0})<\omega(q_0,p_0)<-H_{11}(q_0,q_{1}).
\end{equation}
	In particular, this means that if the inequality \[H_{22}(q_{-1},q_{0})<-H_{11}(q_0,q_{1})\] is violated, then the orbit of $(q_0,p_0)$ is not an m-orbit.

	Since we are concerned with m-orbits we shall look at the points $(q_0,p_0)$ where it holds that
	
\begin{equation}\label{eq:cone}
	H_{22}(q_{-1},q_{0})<-H_{11}(q_0,q_{1}).
\end{equation}
	
	\begin{definition}\label{def:2211}
		For every point $(q_0,p_0)$ satisfying \eqref{eq:cone} the tangent lines with the slopes $-H_{11}(q_0,q_{1})$ and $H_{22}(q_{-1},q_{0})$ divide the tangent plane $T_{(q_0,p_0)}\mathbb {A}$ into four cones which we denote by $N_H,W_H,S_H,E_H$ in the clockwise direction, where $N_H$ contains the vertical vector $\frac{\partial}{\partial p}$.
		See Figure \refeq{fig:coneCondition}.
	\end{definition}

	\subsection{Two generating functions.}\label{subsection:twogenfun}
In this subsection we prove Theorem \ref{thm:equalMOrbits}.
	Suppose that the map $T$ is a twist map with respect to two sets of symplectic coordinates $(q,p)$ and $(x,y)$, with the generating functions $H(q,q'), G(x,x')$.
	We assume that the twist condition $H_{12},G_{12}>0$, and the geometric assumption \eqref{eq:assumption} hold.


		Let us show, for example, the inclusion $\mathcal{M}_H\subseteq\mathcal{M}_G$.
		 Take any m-orbit in $\mathcal{M}_H$, and let $z$ be a point of that orbit.
		 Write $(q_0,p_0)$ for the $(q,p)$ coordinates of $z$.
		By Theorem \ref{thm:mconfigPosJacobi}, the corresponding m-configuration has a positive Jacobi field $\delta q_n$.
		 Hence there exists a non-vertical $T$-invariant vector field $(\delta q_n, \delta p_n)$ along the m-orbit, with $\delta q_n>0$. 
		  This inequality, together with the bounds \eqref{eq:bounds}, imply that this vector field lies in the cone $E_H$.
		  By the assumption \eqref{eq:assumption}, the vector $\frac{\partial}{\partial y}$ is in $N_H$.
		  This means that the vector pair $(\delta q_n,\delta p_n)$, $\frac{\partial}{\partial y}$ is oriented positively.
		    Denote the coordinates of the Jacobi field with respect to the coordinates $(x,y)$ by $(\delta x_n, \delta y_n)$.
		    Since $(q,p)$ and $(x,y)$ are two symplectic coordinates, it follows that the transition preserves orientation.
		    Hence the vectors $(\delta x_n, \delta y_n)$ and the vertical vector are also oriented positively, which means that $\delta x_n>0$.
		    The vector field $(\delta x_n, \delta y_n)$ is again $T$-invariant, and hence its projection $\set{\delta x_n}$ is a positive Jacobi field also for the generating function $G$.
		    Hence by the criterion the orbit of $z$ belongs to $\mathcal{M}_G$. 
		This completes the proof.
	
	\subsection{Application to Birkhoff billiards}\label{subsetcion:applyToBilliards}
Let $\gamma$ be a planar strictly convex $C^2$ smooth curve.
	The phase cylinder $\mathbb{A}$, of all oriented lines intersecting the billiard table $\gamma$, can be endowed with two sets of symplectic coordinates as follows:
	consider an oriented line incoming the billiard table at the point $\gamma(s)$ with the angle $\delta$, where $s$ is the arc-length parameter on $\gamma$. 
	Let $\varphi$ be the angle between the right unit normal to the line and the horizontal direction, and $p$ be the signed distance from the origin to the line, see Figure \ref{fig:parameterPhaseSpaceBirk}.
	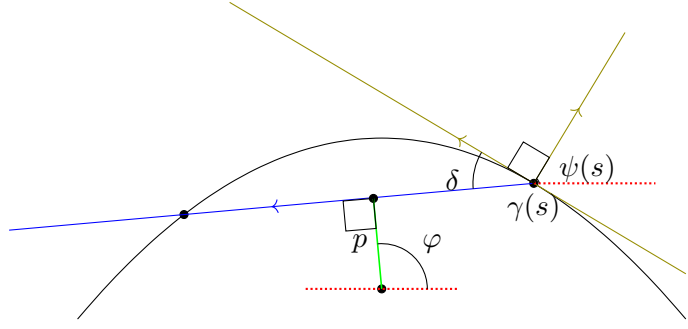
\begin{figure}
\begin{center}
\begin{tikzpicture}[scale = 2]
\draw[domain = -2:2, smooth, variable = \x, black] plot({\x},{1-0.3*\x*\x});
\tkzDefPoint(1,0.7){A};
\tkzDefPoint(-1.3,0.493){B};
\tkzDrawPoints(A,B);
\begin{scope}[decoration={
markings,
mark = at position 0.5 with {\arrow{>}}}
]
\draw[postaction = {decorate},blue](1,0.7)--(-1.3-0.5*2.3,0.493-0.5*0.207);
\draw[postaction = {decorate},olive](2,0.1)--(-1,1.9);
\draw[postaction = {decorate},olive](1,0.7)--(1.6,1.7);
\end{scope}
\tkzDefPoint(-1,1.9){C};
\tkzMarkAngle[size = 0.4, mark = none](C,A,B);
\node[left] at (0.58,0.72) {$\delta$};
\node[below] at(1,0.7) {$\gamma(s)$};
\draw[densely dotted,thick,red] (1,0.7) -- (1.8,0.7);
\tkzDefPoint(1.6,1.7){D}
\tkzDefPoint(1.8,0.7){E}
\tkzMarkAngle[size = 0.4, mark = none, arc = ll](E,A,D);
\node[right] at (1.1,0.8) {$\psi(s)$};
\tkzDefPoint(2,0.1){F}
\tkzMarkRightAngle[size = 0.2](D,A,C)
\tkzDefPoint(0,0){O};
\tkzDrawPoint(O);
\tkzDefPoint(-0.0909*0.6,1*0.6){P};
\tkzDrawPoint(P);
\tkzDrawSegment[green](O,P);
\tkzMarkRightAngle[size = 0.2](O,P,B);
\node[left] at (-0.0909*0.3,0.3) {$p$};
\draw[densely dotted, thick, red] (-0.5,0)--(0.5,0);
\tkzDefPoint(0.5,0){J};
\tkzMarkAngle[size = 0.3,mark = none](J,O,P);
\node[right] at (.2,0.3) {$\varphi$};
\end{tikzpicture}
\caption{Parametrization of the phase space of Birkhoff billiards, with respect to the two symplectic coordinates.}\label{fig:parameterPhaseSpaceBirk}
\end{center}
\end{figure}
	
The generating functions with respect to these two sets of symplectic coordinates are 

\[
L(s,s_1)=|\gamma(s_1)-\gamma(s)| \ \textrm{ and, } \ S(\varphi,\varphi_1)=2h(\psi)\sin\delta ,\]

	where $\psi:=\frac{\varphi+\varphi_1}{2}, \ \delta:=\frac{\varphi_1-\varphi}{2}$.
	\begin{theorem}\label{thm:equalityForBirkhoff}
	For planar Birkhoff billiards we have:
		$\mathcal{M}_L=\mathcal{M}_S$.
	\end{theorem}
	
\begin {proof}
The change of variables $(s,\cos\delta)_{\overrightarrow{F}} (\varphi, p)$ is given by the formulas
	\[
	\begin{cases}
	 \varphi=\psi+\delta,\\
	 p=h(\psi)\cos\delta+h'(\psi)\sin\delta,
	\end{cases}
	\] where $\psi=\int k(s) ds$, and $k$ is the curvature of $\gamma$. 
	Then the matrix of the differential  is:
	
	\[DF=
	\begin{pmatrix}
		k(s)& -\frac{1}{\sin\delta} \\
	k(s)(	h'(\psi)\cos\delta+h''(\psi)\sin\delta) & (	h(\psi)\sin\delta-h'(\psi)\cos\delta)\frac{1}{\sin\delta}
	\end{pmatrix}.
	\]
	We need to check the validity of the geometric assumption \eqref{eq:assumption}.
	Suppose $z$ is a point of an m-orbit for $S$.
	We compute at the tangent space to $z$:
	
\[
DF\left(\frac{\partial}{\partial (\cos\delta)}\right)=\begin{pmatrix}
a \\
b \\
\end{pmatrix}:=
\begin{pmatrix}
-\frac{1}{\sin\delta}\\
	(h(\psi)\sin\delta-h'(\psi)\cos\delta)\frac{1}{\sin\delta} \\
\end{pmatrix}.
\]
Recall the second order partial derivatives of $S$:
	\begin{equation}\label{eq:secondOrderS}
\begin{cases}	S_{11}(\varphi_0,\varphi_1)=\frac{1}{2}(h''(\psi)-h(\psi))\sin\delta-h'(\psi)\cos\delta,\\
	S_{22}(\varphi_0,\varphi_1)=\frac{1}{2}(h''(\psi)-h(\psi))\sin\delta+h'(\psi)\cos\delta,\\
	S_{12}(\varphi_0,\varphi_1)=\frac{1}{2}(h''(\psi)+h(\psi))\sin\delta.
	\end{cases}
\end{equation}
Write $\rho(\psi)=h(\psi)+h''(\psi)$ for the radius of curvature of $\gamma$ at the point the normal makes an angle of $\psi$ with the $x$-axis.
Then,
\begin{equation}\label{eq:ab}
\frac{b}{a}=-h(\psi)\sin\delta+h'(\psi)\cos\delta=S_{22}-\frac{1}{2}\rho(\psi)\sin(\delta)<S_{22},
\end{equation}
and in addition 
\begin{equation}\label{eq:b}
a=	-\frac{1}{\sin\delta}<0.
\end{equation}
Inequalities \eqref{eq:ab} and \eqref{eq:b} imply that the vector $\frac{\partial}{\partial (\cos\delta)}$ belongs to the cone $N_S$.
This is because inequality \eqref{eq:cone} implies that the line with slope $S_{22}$ is the line with the smaller slope, and hence the vector $(a,b)$ is above this line in the left half plane, so it is in $N_S$.

In the other direction, suppose $z$ is a point of an m-orbit of $L$, and let 
\[
DF^{-1}\left(\frac{\partial}{\partial p}\right)=:\begin{pmatrix}
a \\
b \\
\end{pmatrix}.
\]
Then we have
\begin{equation}\label{eq:DF}
DF  \begin{pmatrix}
a \\
b \\
\end{pmatrix}=  \begin{pmatrix}
0 \\
1 \\
\end{pmatrix}.
\end{equation}
Hence we get from the first row of \eqref{eq:DF}
\[ak(s)-\frac{1}{\sin \delta} b=0.\]
Therefore, 
\begin{equation}\label{eq:ab2}
\frac{b}{a}=k(s)\sin\delta.
\end{equation}
Now, if we use the explicit formulas for the second derivatives of $L$ (see \cite{Bialy1993}) we have
\begin{equation}\label{eq:ab3}
-L_{11}=k(s)\sin\delta-\frac{\sin^2 \delta}{L}<\frac{b}{a}.
\end{equation}
Moreover, we compute from the second row of \eqref{eq:DF}:
\[ak(s)(	h'(\psi)\cos\delta+h''(\psi)\sin\delta)+ (	h(\psi)\sin\delta-h'(\psi)\cos\delta)\frac{1}{\sin\delta}b=1.\]
Using \eqref{eq:ab2} we get 
\begin{equation}\label{eq:b2}
a=\frac{1}{\sin\delta}>0.
\end{equation}
It then follows from \eqref {eq:ab3}, \eqref {eq:b2}  that the vector $\left(\frac{\partial}{\partial p}\right)$  belongs to the cone  $N_L$.
This again follows from the fact that $z$ is a point of an m-orbit, and from inequality \eqref{eq:cone}, the line with slope $-L_{11}$ has the larger slope of the two lines that determine the cone.
\end{proof}

\section{Effective bounds for integrability of billiards and $L^2$ norm}\label{section:effectiveBounds}
Our goal in this section is to give effective bounds, similar to those given in \cite{BialyMisha2015EbiE} for the rigidity of integrable Birkhoff billiards.
We prove Theorems \ref{thm:effectiveGeneralCurve}, \ref{thm:estimate-beta}.
The bound of Theorem \ref{thm:effectiveGeneralCurve} is  for arbitrary planar strictly convex $C^2$ smooth curves $\gamma$.
We derive another effective bound for the rigidity of billiards in such curves, in addition to the two derived in \cite{BialyMisha2015EbiE} (see inequalities (1.1) and (1.2) there).
\begin{proof}[Proof of Theorem \ref{thm:effectiveGeneralCurve}]
As explained in Subsection \ref{subsection:omegaBounds}, one can construct a function $\omega$ such that along any m-orbit, equation \eqref{eq:omega1} holds true. 
By Theorem \ref{thm:equalityForBirkhoff}, the m-orbits for both the generating functions for Birkhoff billiards are the same, so we can choose to work with the generating function in the $(\varphi,p)$ coordinates.
If $(\varphi_0,p_0)$ are the coordinates of a line that is a part of an m-orbit, then there exist a function $\omega$ and a positive function $\nu_1$ such that \eqref{eq:omega1} holds:
\[	\begin{cases}
	\omega(T(\varphi_0,p_0))=S_{22}(\varphi_0,\varphi_1)+S_{12}(\varphi_0,\varphi_1)\nu_1(\varphi_0,p_0)^{-1},\\
	\omega(\varphi_0,p_0)=-S_{11}(\varphi_0, \varphi_1)-S_{12}(\varphi_0,\varphi_1)\nu_{1}(\varphi_0,p_0).
	\end{cases}\]
	By subtracting we get
	\begin{gather*}\omega(T(\varphi_0,p_0))-\omega(\varphi_0,p_0)=S_{11}(\varphi_0,\varphi_1)+S_{22}(\varphi_0,\varphi_1)+ \\
+S_{12}(\varphi_0,\varphi_1)\Big(\nu_1(\varphi_0,p_0)+\nu_1(\varphi_0,p_0)^{-1}\Big)	\geq \\
\geq S_{11}(\varphi_0,\varphi_1)+S_{22}(\varphi_0,\varphi_1)+2S_{12}(\varphi_0,\varphi_1).
	\end{gather*}
	where we used the fact that $S_{12}$ and $\nu_1$ are positive.
	Now we integrate both sides on the $T$ invariant set $\mathcal{M}$ with respect to the $T$ invariant measure
	\[d\mu=\frac{1}{4}(h(\psi)+h''(\psi))\sin\delta d\psi d\delta.\]
	The integral of the left hand side vanishes, and we are left with
	\[\int\limits_{\mathcal{M}}(S_{11}+S_{22}+2S_{12})d\mu \leq 0.\]
Using \eqref{eq:secondOrderS} and simplifying, we get
\[\int\limits_{\mathcal{M}} 2h''(\psi)\sin\delta d\mu \leq 0.\]
Since $\mathcal{M}=\mathbb{A}\setminus\Delta$, after dividing by $2$, we get
\begin{equation}\label{eq:IntInequalityForRigidity}
\int\limits_{\mathbb{A}}h''(\psi)\sin\delta d\mu \leq \int\limits_{\Delta} h''(\psi)\sin\delta d\mu.
\end{equation}
We give an upper bound for the right hand side, and a lower bound on the left hand side, and together we get the required bound.
For the right hand side, write:
\begin{gather*}
\int\limits_{\Delta} h''(\psi)\sin\delta d\mu \leq \Big| \int\limits_{\Delta} h''(\psi)\sin\delta d\mu \Big|\leq \\
\leq \int\limits_{\Delta}|h''(\psi)\sin\delta| d\mu \leq \mu(\Delta)\max\limits_{\Delta} |h''|.
\end{gather*}
It holds that $h(\psi)+h''(\psi)=\rho(\psi)$, where $\rho(\psi)$ is the radius of curvature at the point where the normal to $\gamma$ makes an angle of $\psi$ with the $x$-axis, so \[|h''|\leq\rho+\max h.\]
Since $h(\psi)+h(\psi+\pi)$ is the width of $\gamma$ in the direction $\psi$, and the maximal width is the diameter of $\gamma$, we have $\max h\leq D$, where $D$ is the diameter.
Also, the maximal radius of curvature of $\gamma$ is $\frac{1}{\beta}$ where $\beta$ is the minimal curvature of $\gamma$.
This gives us the estimate
\begin{equation}\label{eq:RHSEffectiveBound}
\int\limits_{\Delta} h''(\psi)\sin\delta d\mu \leq \Big(D+\frac{1}{\beta}\Big)\mu(\Delta)\leq \frac{3}{\beta}\mu(\Delta),
\end{equation}
where we used Blaschke's rolling disk theorem, stating that $\gamma$ is contained inside a disk with radius equal to the maximal radius of curvature of $\gamma$, and this means that $D\leq\frac{2}{\beta}$.

Now we turn to the left hand side.
\begin{gather*}
\int\limits_{\mathbb{A}}h''(\psi)\sin\delta d\mu=\frac{1}{4}\int\limits_0^{2\pi}\int\limits_0^\pi h''(\psi)(h''(\psi)+h(\psi))\sin^2\delta d\psi d\delta =\\
=\frac{1}{4}\int\limits_0^\pi \sin^2\delta d\delta\int\limits_0^{2\pi} h''(\psi)(h''(\psi)+h(\psi))d\psi=\frac{\pi}{8}\int\limits_0^{2\pi}(h''(\psi))^2+h''(\psi)h(\psi) d\psi.
\end{gather*}
From integration by parts it follows that
\[\int\limits_0^{2\pi}h(\psi)h''(\psi)d\psi=-\int\limits_0^{2\pi}(h'(\psi))^2 d\psi,\]
and as a result
\[\int\limits_{\mathbb{A}}h''(\psi)\sin\delta d\mu=\frac{\pi}{8}\int\limits_0^{2\pi}(h''(\psi))^2-(h'(\psi))^2 d\psi.\]
Write the Fourier expansion of $h$, $h(\psi)=\sum\limits_{n\in\Z} c_n e^{in\psi}$.
We have $h'(\psi)=\sum\limits_{n\in\Z} inc_n e^{in\psi}$ and $h''(\psi)=\sum\limits_{n\in\Z} -n^2 c_n e^{in\psi}$.
We now use Parseval's identity.
\begin{gather*}
\int\limits_0^{2\pi}(h''(\psi))^2-(h'(\psi))^2d\psi=2\pi\sum\limits_{n\in\Z} n^4 |c_n|^2-n^2 |c_n|^2=2\pi\sum\limits_{n\in\Z} n^2(n^2-1)|c_n|^2 = \\
= 2\pi\sum\limits_{|n|\geq 2} n^2(n^2-1)|c_n|^2 \geq 2\pi\sum\limits_{|n|\geq 2} 12 |c_n|^2 = \\
=12\int\limits_0^{2\pi}|h(\psi)-c_0-c_1e^{i\psi}-c_{-1}e^{-i\psi}|^2 d\psi =24\pi d^2(h,W).
\end{gather*}
Here $W$ is the subspace of $L^2[0,2\pi]$ spanned by $\set{1,\cos(\psi),\sin(\psi)}$, and $d(\cdot,W)$ denotes the $L^2$-distance to that subspace.
We now have the following estimate
\begin{equation}\label{eq:LHSLowerBoundEffective}
\int\limits_{\mathbb{A}}h''(\psi)\sin\delta d\mu \geq 3\pi^2 d^2(h,W).
\end{equation}
Now we can combine \eqref{eq:IntInequalityForRigidity}, \eqref{eq:RHSEffectiveBound} and \eqref{eq:LHSLowerBoundEffective}, to get
\[3\pi^2 d^2(h,W)\leq \frac{3}{\beta}\mu(\Delta).\]
This gives us the required inequality
\[\mu(\Delta)\geq\pi^2\beta d^2(h,W).\]
This bound is indeed sharp for circles:
if $\mu(\Delta)=0$ then $h\in W$, which means that there exist constants $a,b,c$ for which $h(\psi)=a+b\cos\psi+c\sin\psi$, and this is the support function of a circle (with center $(b,c)$ and radius $a$).
\end{proof}
The result we obtained is very similar to the bound (1.2) of \cite{BialyMisha2015EbiE}. 
The main difference is that the isoperimetric defect that appears there is replaced here with the $L^2$ distance from $h$ to the subspace $W$.
It is interesting if these bounds are actually comparable in some way.

We now turn to Theorem \ref{thm:estimate-beta}.
Let us denote by $\mathcal{C}$ the set of centrally symmetric, strictly convex, $C^2$ smooth curves, for which the billiard map has an invariant curve of $4$-periodic orbits, as was considered in \cite{bialy2020birkhoffporitsky}.
For a centrally symmetric curve $\gamma$, let $h:[0,2\pi]\to\R$ denote its support function (with respect to the center of symmetry).
It was shown in \cite{bialy2020birkhoffporitsky} that 
\[\mathcal{C}\subseteq \bigset{\gamma\mid h^2(\psi)=c_0+\sum\limits_{n\in 2+4\Z} c_n e^{in\psi}, h + h''>0}.
\]
The condition $h+h''>0$ implies that $\gamma$ is strictly convex, since $h(\psi)+h''(\psi)$ is the radius of curvature of $\gamma$, at the point where the normal to $\gamma$ makes an angle $\psi$ with the horizontal direction.
First we show that the converse inclusion also holds.
This means that the class of curves $\mathcal{C}$ is a rather ``large" set in $L^2[0,2\pi]$.
\begin{proposition}\label{prop:sufficientCondition}
It holds that \begin{equation}\label{eq:classOfCurves}
\mathcal{C}=\bigset{\gamma\mid h^2(\psi)=c_0+\sum\limits_{n\in 2+4\Z} c_n e^{in\psi}, h+h'' > 0}.
\end{equation}
More precisely, given a $C^2$ function $h:[0,2\pi]\to\R$ satisfying both conditions in the definition of the class $\mathcal{C}$, there exists a curve $\gamma\in\mathcal{C}$ for which the support function is $h$.
\end{proposition}
\begin{proof}
Let $h$ be a function satisfying those conditions, and define the curve $\gamma$ by 
\begin{equation}\label{eq:curveAndSupport}
\gamma(\psi)=h(\psi)\begin{pmatrix}
\cos\psi \\
\sin\psi 
\end{pmatrix}
+h'(\psi)\begin{pmatrix}
-\sin\psi \\
\cos\psi
\end{pmatrix}.\end{equation}
It is well known that a curve $\gamma$ and its support function are related by this equation, and hence the support function of $\gamma$ is $h$.
As before, $\rho(\psi)$ denotes that radius of curvature of $\gamma$ at the point $\gamma(\psi)$, and $\rho(\psi)=h(\psi)+h''(\psi)$.
The fact that $h$ is $C^2$ implies that $\rho$ is continuous, and hence the curve $\gamma$ itself is $C^2$.
The condition $h+h''>0$ then means that $\rho > 0$, which implies that $\gamma$ is strictly convex. 
Next, the fact that \[h^2(\psi)=c_0+\sum\limits_{n\in 2+4\Z} c_n e^{in\psi}\] implies that $h^2(\psi+\pi)=h^2(\psi)$.
Since $h$ is positive it also follows that $h(\psi+\pi)=h(\psi)$, and from this and \eqref{eq:curveAndSupport} we get that $\gamma(\psi+\pi)=-\gamma(\psi)$, which means that $\gamma$ is centrally symmetric. 
Next, the Fourier decomposition of $h^2$ also yields the identity  $h^2(\psi)+h^2(\psi+\frac{\pi}{2})=2c_0$, for all $\psi$.
Since $c_0=h^2(0)$ is positive, we can find $R>0$ for which: \[h^2(\psi)+h^2(\psi+\frac{\pi}{2})=R^2.\]
Therefore, we can find a function $d:[0,2\pi]\to[0,\frac{\pi}{2}]$ for which
\begin{equation}\label{eq:definitionOfD}
\begin{cases}
h(\psi)=R\sin d(\psi) , \\
h(\psi+\frac{\pi}{2})=R\cos d(\psi) . \\
\end{cases}
\end{equation}
The range $[0,\frac{\pi}{2}]$ can be chosen for $d$ since $h$ is positive.
We show that this $d(\psi)$ is an invariant curve of $4$-periodic billiard orbits in $\gamma$, namely, that for all $\psi$, if we consider the line that passes through $\gamma(\psi)$ and makes an angle of $d(\psi)$ with the tangent, then this line determines a $4$-periodic billiard orbit.
For that end, it is enough to show that the line passing through $\gamma(\psi)$, making an angle of $d(\psi)$ with the tangent, intersects $\gamma$ at the point $\gamma(\psi+\frac{\pi}{2})$ and makes an angle of $d(\psi+\frac{\pi}{2})$ with the tangent at that point. 
Observe that the exterior normal to $\gamma(\psi)$ makes an angle of $\psi$ with the $x$-axis, and hence the tangents to $\gamma(\psi)$ and $\gamma(\psi+\frac{\pi}{2})$ are necessarily orthogonal. 
Also, it holds that \[d(\psi+\frac{\pi}{2})=\arcsin\frac{h(\psi+\frac{\pi}{2})}{R}=\arcsin\frac{R\cos d(\psi)}{R}=\frac{\pi}{2}-d(\psi).\]
As a result, if the segment from $\gamma(\psi)$ to $\gamma(\psi+\frac{\pi}{2})$ makes an angle of $d(\psi)$ with the tangent at $\gamma(\psi)$ it then also must make an angle of $d(\psi+\frac{\pi}{2})$ with the tangent to $\gamma(\psi+\frac{\pi}{2})$, so it is enough to verify the first claim.
Consequently, we only need to verify that the vector $\gamma(\psi+\frac{\pi}{2})-\gamma(\psi)$ is parallel to the vector \[v=(\cos(\psi+\frac{\pi}{2}+d(\psi)),\sin(\psi+\frac{\pi}{2}+d(\psi)))=(-\sin(\psi+d(\psi)),\cos(\psi+d(\psi))).\]
Therefore we compute:
\begin{gather*}
\gamma(\psi+\frac{\pi}{2})-\gamma(\psi)=h(\psi+\frac{\pi}{2})\begin{pmatrix}
\cos (\psi+\frac{\pi}{2}) \\
\sin (\psi+\frac{\pi}{2})
\end{pmatrix}
+h'(\psi+\frac{\pi}{2})\begin{pmatrix}
-\sin (\psi+\frac{\pi}{2}) \\
\cos (\psi+\frac{\pi}{2})
\end{pmatrix}\\
-h(\psi) \begin{pmatrix}
\cos (\psi) \\
\sin (\psi) 
\end{pmatrix}
-h'(\psi)\begin{pmatrix}
-\sin (\psi) \\
\cos (\psi)
\end{pmatrix}.
\end{gather*}
Consider first the terms with $h(\psi+\frac{\pi}{2})$ and $h(\psi)$.
Since $h(\psi+\frac{\pi}{2})=R\cos d(\psi)$ and $h(\psi)=R\sin d(\psi)$, we have 
\begin{gather*}
u := h(\psi+\frac{\pi}{2})
\begin{pmatrix}
\cos (\psi+\frac{\pi}{2}) \\
\sin (\psi+\frac{\pi}{2})
\end{pmatrix}
 - h(\psi) \begin{pmatrix}
\cos (\psi) \\
\sin (\psi) 
\end{pmatrix}
 = \\
 =R\bigg(\cos d(\psi)\begin{pmatrix}
-\sin (\psi) \\
\cos (\psi)
\end{pmatrix}
 - \sin d(\psi)\begin{pmatrix}
\cos (\psi) \\
\sin (\psi)
\end{pmatrix}\bigg)=\\
=R\begin{pmatrix}
-\sin(\psi)\cos d(\psi)-\sin d(\psi)\cos(\psi) \\
\cos d(\psi)\cos(\psi)-\sin d(\psi)\sin\psi
\end{pmatrix} = R\begin{pmatrix}
-\sin(\psi+d(\psi)) \\
\cos(\psi+d(\psi))
\end{pmatrix}.
\end{gather*}
So we see that the vector $u$ is indeed parallel to the vector $v$.
Now for the terms with $h'(\psi)$, and $h'(\psi+\frac{\pi}{2})$.
\begin{gather*}
w := h'(\psi+\frac{\pi}{2})\begin{pmatrix}
-\sin (\psi+\frac{\pi}{2}) \\
\cos (\psi+\frac{\pi}{2})
\end{pmatrix} - h'(\psi)\begin{pmatrix}
-\sin (\psi) \\
\cos (\psi)
\end{pmatrix} =\\
= h'(\psi+\frac{\pi}{2})\begin{pmatrix}
-\cos(\psi) \\
-\sin(\psi) 
\end{pmatrix}-h'(\psi)\begin{pmatrix}
-\sin(\psi) \\
\cos(\psi)
\end{pmatrix}=\\
=-h'(\psi+\frac{\pi}{2})\begin{pmatrix}
\cos(\psi) \\
\sin(\psi)
\end{pmatrix} - h'(\psi)\begin{pmatrix}
-\sin(\psi) \\
\cos(\psi)
\end{pmatrix}.
\end{gather*}
Since $h^2(\psi)+h^2(\psi+\frac{\pi}{2})$ is constant, it follows that \[h(\psi)h'(\psi)+h(\psi+\frac{\pi}{2})h'(\psi+\frac{\pi}{2})=0,\] and hence that $h'(\psi+\frac{\pi}{2})=-\frac{h(\psi)}{h(\psi+\frac{\pi}{2})}h'(\psi)=-\tan d(\psi) h'(\psi)$.
We use it in the above formula for $w$:
\begin{gather*}
w=h'(\psi)\tan d(\psi)\begin{pmatrix}
\cos(\psi) \\
\sin(\psi)
\end{pmatrix}-h'(\psi)\begin{pmatrix}
-\sin(\psi) \\
\cos(\psi)
\end{pmatrix}
=\\
=\frac{h'(\psi)}{\cos d(\psi)}\begin{pmatrix}
\sin d(\psi)\cos(\psi)+\cos d(\psi) \sin(\psi) \\
\sin d(\psi) \sin (\psi) - \cos d(\psi)\cos(\psi) 
\end{pmatrix} =\\
= \frac{h'(\psi)}{\cos d(\psi)}\begin{pmatrix}
\sin(\psi+d(\psi)) \\
-\cos(\psi+d(\psi))
\end{pmatrix} = -\frac{h'(\psi)}{\cos d(\psi)} v.
\end{gather*}
So the vector $w$ is also parallel to the vector $v$, and hence $\gamma(\psi+\frac{\pi}{2})-\gamma(\psi)=u+w$ is also parallel to $v$, finishing the proof.
\end{proof}

Now we are in position to prove Theorem \ref{thm:estimate-beta}.
\begin{proof}[Proof of Theorem \ref{thm:estimate-beta}]
The beginning of this proof is identical to the first paragraph of the proof of Theorem \ref{thm:effectiveGeneralCurve}.
We can then start from equality \eqref{eq:omega1}, and continue as in \cite[Section 5]{bialy2020birkhoffporitsky}, to get an analogue of the inequality (18) in \cite{bialy2020birkhoffporitsky}.
For completeness, we shall reconstruct the process.
Let $(\varphi_0,p_0)$ be the coordinates of an oriented line which is a part of an m-orbit.
Then \eqref{eq:omega1} implies that there exist functions $\omega$ and $\nu_1$ such that
\[	\begin{cases}
	\omega(T(\varphi_0,p_0))=S_{22}(\varphi_0,\varphi_1)+S_{12}(\varphi_0,\varphi_1)\nu_1(\varphi_0,p_0)^{-1},\\
	\omega(\varphi_0,p_0)=-S_{11}(\varphi_0, \varphi_1)-S_{12}(\varphi_0,\varphi_1)\nu_{1}(\varphi_0,p_0).
	\end{cases}\]
	where $T$ is the billiard map, and $(\varphi_1,p_1)$ are the coordinates of the line $T(\varphi_0,p_0)$.
	Since $p$ is the momentum coordinate, then \[p_0=-S_1(\varphi_0,\varphi_1)=h(\psi)\cos\delta-h'(\psi)\sin\delta\] and \[p_1=S_2(\varphi_0,\varphi_1)=h(\psi)\cos\delta+h'(\psi)\sin\delta,\] where $\psi=\frac{\varphi_0+\varphi_1}{2}$ and $\delta=\frac{\varphi_1-\varphi_0}{2}$.
	We multiply the first equation by $p_1^2$, and the second by $p_0^2$, and subtract
	\begin{gather*}
	p_1^2\omega(T(\varphi_0,p_0))-p_0^2\omega(\varphi_0,p_0)= \\
	=p_0^2S_{11}(\varphi_0,\varphi_1)+p_1^2 S_{22}(\varphi_0,\varphi_1)+S_{12}\Big(p_0^2\nu_1(\varphi_0,p_0)+p_1^2\nu_1(\varphi_0,p_0)^{-1}\Big) \geq \\
	\geq p_0^2 S_{11}(\varphi_0,\varphi_1)+p_1^2S_{22}(\varphi_0,\varphi_1)+2p_0p_1S_{12}(\varphi_0,\varphi_1).
	\end{gather*}
	We used the fact that $\nu_1$ and $S_{12}$ are positive.
	Now integrate both sides of this inequality on the invariant set $\mathcal{M}\cap\mathcal{A}$, with respect  to the invariant measure \[d\mu=\frac{1}{4}(h(\psi)+h''(\psi))\sin\delta d\delta d\psi.\]
	Since this set is invariant under $T$, and $T$ is measure preserving, the integral in the left hand side vanishes.
	We are left with
	\[\int_{\mathcal{M}\cap \mathcal{A}} \Big(p_0^2 S_{11}(\varphi_0,\varphi_1)+p_1^2 S_{22}(\varphi_0,\varphi_1)+2p_0p_1 S_{12}(\varphi_0,\varphi_1)\Big)d\mu \leq 0.\]
	So now we substitute the expressions for $p_0$ and $p_1$ and the second order derivatives of $S$ computed above \eqref{eq:secondOrderS}.
Then after simplification, we get the following inequality
 
\begin{equation}\label{eq:integralInequalityExp}
0\geq\int\limits_{\mathcal{M}\cap\mathcal{A}}\Big[\cos^2\delta\sin\delta\Big(h''h^2+3h(h')^2\Big)-h(h')^2\sin\delta\Big]d\mu.
\end{equation}
Call the first summand of the integrand $A$, and the second one $B$. 
Then inequality \eqref{eq:integralInequalityExp} reads:
\[\int\limits_{\mathcal{M}\cap\mathcal{A}} Ad\mu \leq \int\limits_{\mathcal{M}\cap\mathcal{A}}B d\mu.\]
The function $B=h(h')^2\sin\delta$ is non-negative, so \[\int\limits_{\mathcal{M}\cap\mathcal{A}}Bd\mu\leq\int\limits_{\mathcal{A}}Bd\mu.\]
In the left hand side, write \[\int\limits_{\mathcal{M}\cap \mathcal{A}} Ad\mu = \int\limits_{\mathcal{A}} Ad\mu-\int\limits_{\Delta} Ad\mu,\] and then one gets
\begin{equation}\label{eq:integralInequalityShort}
\int\limits_{\mathcal{A}}(A-B)d\mu\leq\int\limits_{\Delta} Ad\mu.
\end{equation}
Left hand side of the last inequality is exactly the right hand side of the inequality (18) in \cite{bialy2020birkhoffporitsky}, but multiplied by $\frac{1}{4}$ (this constant is omitted in \cite{bialy2020birkhoffporitsky}, since the integral is compared to zero). 
Hence it can be simplified by Lemma 5.1 of \cite{bialy2020birkhoffporitsky}:
\[\int\limits_{\mathcal{A}} (A-B) d\mu=\frac{\pi R^4}{1024}\int\limits_0^\pi(\mu'')^2-4(\mu')^2 d\psi ,\]
where $\mu(\psi)=\cos(2d(\psi))$.
Now we bound this integral from below, and bound $\int\limits_\Delta Ad\mu$ from above. 
It holds that \[\mu(\psi)=\cos(2d(\psi))=1-2\sin^2 d(\psi)=1-2\frac{h^2(\psi)}{R^2}.\]
As a result, $\mu'(\psi)=-\frac{2}{R^2}(h^2)'$, and $\mu''(\psi)=-\frac{2}{R^2}(h^2)''$.
Thus, we have \[\int\limits_0^\pi(\mu'')^2-4(\mu')^2d\psi=\frac{4}{R^4}\int\limits_0^\pi\Big((h^2)''\Big)^2-4\Big((h^2)'\Big)^2d\psi.\]
Since $\gamma$ is a curve in $\mathcal{C}$, the Fourier expansion of $h^2$ is as in equation \eqref{eq:classOfCurves}.
Now use Parseval's identity in $L^2[0,\pi]$:
\begin{gather*}
\int\limits_0^\pi \Big((h^2)''\Big)^2-4\Big((h^2)'\Big)^2d\psi=\pi\sum\limits_{n\in 2+4\Z} (n^4-4n^2)|c_n|^2=\\
=\pi\sum\limits_{\tiny{\begin{matrix}
n\in 2+4\Z \\
|n|>2
\end{matrix}}} (n^4-4n^2)|c_n|^2
\geq\pi\sum\limits_{\tiny{\begin{matrix}
n\in 2+4\Z \\
|n|>2
\end{matrix}}} 1000|c_n|^2,
\end{gather*}
where we used the fact that for $|n|\geq 6$, $n^4-4n^2\geq 1000$. 
Now use Parseval's identity again
\begin{multline*}
\int\limits_0^\pi \Big((h^2)''\Big)^2-4\Big((h^2)'\Big)^2d\psi\geq 1000\int\limits_0^\pi |h^2(\psi)-c_0-c_2e^{2i\psi}-c_{-2}e^{-2i\psi}|^2d\psi\geq \\
\geq 1000\pi d^2(h^2,\mathrm{span}\set{1,\cos(2\psi),\sin(2\psi)}),
\end{multline*}
where $d$ denotes the distance in the $L^2$ norm between the function $h^2$ and a subspace of $L^2[0,\pi]$.
Denote the subspace of $L^2[0,\pi]$ spanned by $\set{1,\cos(2\psi),\sin(2\psi)}$ by $U$. 
Then we get that \[\int\limits_0^\pi(\mu'')^2-4(\mu')^2d\psi\geq\frac{4000\pi}{R^4} d^2(h^2,U).\]
As a result, we get the following lower bound:
\begin{equation}\label{eq:lowerBoundL2}
\int\limits_{\mathcal{A}} A-B d\mu\geq\frac{\pi R^4}{1024}\cdot \frac{4000\pi}{R^4} d^2(h^2,U)=\frac{125\pi^2}{32}d^2(h^2,U).
\end{equation}
Now we turn to finding an upper bound for \[\int\limits_{\Delta} Ad\mu.\]
If $N$ is an upper bound on $A$, then \[\int\limits_{\Delta} Ad\mu\leq N\mu(\Delta),\] so it is enough to find an upper bound for $A$. We have
\[|A|=|\sin\delta\cos^2\delta (h''h^2+3h(h')^2)|\leq |h''|h^2+3h(h')^2.\]Since $\gamma$ is centrally symmetric, $h(\psi)$ is half the width in the direction $\psi$ (recall that in this case $h$ is the support function with respect to the center of symmetry). 
The maximal width is in the direction of the diameter, so applying Blaschke rolling disc theorem again, we get 
$$
h\leq \frac{D}{2}\leq\frac{1}{\beta},
$$where $\beta$ is the minimal curvature of $\gamma$.
Next, since $h+h''=\rho$, then \[|h''|\leq\rho+h\leq\frac{1}{\beta}+\frac{D}{2}\leq \frac{2}{\beta}.\] 

Since $|{\gamma(\psi)}|^2=h(\psi)^2+h'(\psi)^2$, then
 \[h'(\psi)^2\leq \Big(\frac{D}{2}\Big)^2\leq\frac{1}{\beta^2} .\]
Now we put everything together
$$
|A|\leq \frac{5}{\beta^3}.
$$

Thus we get
\begin{equation}\label{eq:boundOnIntA}
	\int\limits_{\Delta} Ad\mu\leq \frac{5}{\beta^3}\mu(\Delta).
\end{equation}
Finally we put together inequalities \eqref{eq:integralInequalityShort}, \eqref{eq:lowerBoundL2} and \eqref{eq:boundOnIntA}, and get
\[\frac{5}{\beta^3}\mu(\Delta)\geq \frac{125\pi^2} {32}d^2(h^2,U) .\]
And hence
\[\mu(\Delta)\geq \frac{25\pi^2} {32}{\beta^3}d^2(h^2,U) ,\]
which is the required inequality.

This bound is sharp for ellipses.
If $\mu(\Delta)=0$, then (since right hand side is also non-negative), it follows that $d(h^2,U)=0$, so $h^2\in\textrm{span}\set{1,\cos(2\psi),\sin(2\psi)}$.
It then follows that $h$ is the support function of an ellipse, which shows that $\mu(\Delta)=0$ can happen only for ellipses. 
\end{proof}
It should be noted that if $h^2(\psi)=c_0+\sum\limits_{n\in 2+4\Z} c_n e^{in\psi}$, then the minimal distance of $h^2$ from the subspace $U$ is realized at the ellipse for which the support function $\tilde{h}$ satisfies $\tilde{h}^2(\psi)=c_0+c_2 e^{2i\psi}+c_{-2}e^{-2i\psi}$.
In that sense, this estimate relates the measure of $\Delta$ and the $L^2$ distance between the square of the support function of $\gamma$, and the square of the support function of the ellipse that ``best approximates $\gamma$".

\bibliography{bibliography}
\bibliographystyle{abbrv}

\end{document}